\documentclass{svproc}
\usepackage{url}

\usepackage{amsmath,amsfonts,amssymb,amsbsy, amsthm,ae,aecompl}
\usepackage{algorithm,algpseudocode}
\usepackage[utf8]{inputenc}
\usepackage[english]{babel}
\usepackage{bm}
\usepackage{color,soul}
\usepackage[noadjust]{cite}
\usepackage{epsfig}
\usepackage{enumerate}
\usepackage{float}
\usepackage{fancyhdr}
\usepackage[T1]{fontenc}
\usepackage[acronym,toc,shortcuts]{glossaries}
 \usepackage[caption=false,font=footnotesize]{subfig}
 \usepackage{soul}
\usepackage{lipsum}
\usepackage{dsfont}
\usepackage{transparent}
\usepackage{tikz,pgfplots}
\usepackage{times}
\usepackage{verbatim}
\usepackage[all]{xy}
\usepackage[colorinlistoftodos]{todonotes}
\usepackage{xcolor}

\makeglossaries
\newacronym{RW}{RW}{random walk}
\newacronym{DFT}{DFT}{discrete Fourier transform}
\newacronym[sort=ell]{test}{LED}{limiting eigenvalue distribution}
\newacronym{RGG}{RGG}{random geometric graph}
\newacronym{DGG}{DGG}{deterministic geometric graph}
\newacronym{RGGs}{RGGs}{random geometric graphs}
\newacronym{ESDF}{ESDF}{ empirical spectral distribution function}
\newacronym{ESDFs}{ESDFs}{ empirical spectral distribution functions}
\newacronym{ER}{ER}{Erd\"{o}s-Rényi}
\newacronym{LSD}{LSD}{limiting spectral distribution}
\newacronym{SD}{SD}{ spectral dimension}
\newacronym{ED}{ED}{ eigenvalue density}

\begin{document}
\mainmatter              
\title{Eigenvalues and Spectral Dimension of Random Geometric Graphs in Thermodynamic Regime}
\titlerunning{Eigenvalues and Spectral Dimension of Random Geometric Graphs in the Thermodynamic Regime}  
%
\author{Konstantin Avrachenkov\inst{1} \and Laura Cottatellucci\inst{2}
\and Mounia Hamidouche\inst{3}\footnote{The authors are listed in the alphabetical order. }}
\authorrunning{Mounia Hamidouche et al.} 
%
\tocauthor{Ivar Ekeland, Roger Temam, Jeffrey Dean, David Grove,
Craig Chambers, Kim B. Bruce, and Elisa Bertino}
\institute{Inria, 2004 Route des Lucioles, Valbonne, France \\
\email{k.avrachenkov@inria.fr}\and
 Department of Electrical, Electronics, and Communication Engineering, FAU, Erlangen, Germany\\
 \email{laura.cottatellucci@fau.de}
\and
Departement of communication systems, EURECOM, Biot, France  \\
  \email{mounia.hamidouche@eurecom.fr}
}
\maketitle              
\setcounter{footnote}{0}
\begin{abstract}
Network geometries are typically characterized by having a finite spectral dimension (SD), $d_{s}$ that characterizes the return time distribution of a random walk on a graph.  The main purpose of this work is to determine the SD of a variety of random graphs called random geometric graphs (RGGs) in the thermodynamic regime, in which the average vertex degree is constant. The spectral dimension depends on the eigenvalue density (ED) of the RGG normalized Laplacian in the neighborhood of the minimum eigenvalues.  In fact, the behavior of the ED in such a neighborhood characterizes the random walk. Therefore, we first provide an analytical approximation for the eigenvalues of the regularized normalized Laplacian matrix of RGGs in the thermodynamic regime. Then,  we show that the smallest non zero eigenvalue converges to zero in the large graph limit. Based on the analytical expression of the eigenvalues,  we show that the eigenvalue distribution in a neighborhood of the minimum value follows a power-law tail. Using this result, we find that the SD of RGGs is approximated by the space dimension $d$ in the thermodynamic regime.
\end{abstract}
\section{Introduction}

The world where we are living in with the phenomena that we observe in daily life is a complex world that scientists are trying to describe via complex networks. Network science has emerged as a fundamental field to study and analyze the properties of complex networks \cite{barabasi2016network}. 
The study of dynamical processes on complex networks is an even more diverse topic. One of the important dynamical processes identified by researchers is the process of diffusion on random geometric structures.  It corresponds to the spread in time and space of a particular phenomenum. This concept is widely used and find application in a wide range of different areas of physics.  For example, in percolation theory, the percolation clusters provide fluctuating geometries \cite{diffusion}. Additional applications are the spread of epidemics \cite{pastor2001epidemic} and the spread of information on social networks \cite{guille2013information}, which are often modeled by random geometries. 
 In particular, the long time characteristics of diffusion provide valuable insights on the average large scale behavior of the studied geometric object. \Gls{SD} is one of the simplest quantities which provides such information.  In 1982, the spectral dimension is introduced for the first time to characterize the low-frequency vibration spectrum of geometric objects \cite{alexander1982density} and then has been widely used in quantum gravity \cite{jonsson1998spectral}.

In a network, in which a particle moves randomly along edges from a vertex to another vertex in discrete steps, the diffusion process can be thought as a stochastic random walk.  Then, \gls{SD} $d_{s}$ is defined in terms of the return probability $\mathrm{P}(t)$ of the diffusion \cite{cooperman2018scaling}

\begin{equation*}
d_{s}=-2 \dfrac{\mathrm{d}\ln \mathrm{P}(t)}{\mathrm{d}\ln (t)},
\end{equation*}
$t$ being the diffusion time.   The spectral dimension, $d_{s}$, defined above is a measure of how likely a random walker return to the starting point after time $t$.  In contrast to the topological dimension, $d_{s}$ need not be an integer. Note that this definition is independent of the particular initial point.

The exact value of $d_{s}$ is only known for a rather limited class of models. For instance, Euclidean lattices in dimension $d$ have spectral dimension $ d_{s}= d$.  In the case of the percolation problem,  Alexander and Orbach conjectured that the spectral dimension of a percolating cluster is $d_{s}= 4/3$ \cite{alexander1982density}. In \cite{durhuus2006random}, the spectral dimension of another random geometry, called random combs is also studied. Random combs are special tree graphs composed of an infinite linear chain to which a number of linear chains are attached according to some probability distribution. In  this particular case, the spectral dimension is found to be $d_{s}=3/2.$
 
The geometric structure considered in this article is the \gls{RGG}.  \glspl{RGG} are models in which the vertices have some random geometric layout and the edges are determined by the position of these vertices. The \gls{RGG}  model used in this work is defined in details in the next section.
 
The main purpose of this work is the study of the spectral dimension and the return-to-origin probability of random walks on \glspl{RGG}. In many applications, an estimator of the spectral dimension can 
serve as an estimator of the intrinsic dimension of the underlying geometric space \cite{pestov2008axiomatic}. In the fields of pattern recognition and machine learning \cite{bishop2006pattern}, the intrinsic dimension of a data set can be thought of as the number of variables needed in a minimal representation of the data. Similarly,    in signal processing of multidimensional signals, the intrinsic dimension of the signal describes how many variables are needed to generate a good approximation of the signal. Therefore, an estimator of the spectral dimension of \glspl{RGG} is relevant and can be used for the estimation of the intrinsic dimension in applications in which the network takes into account the proximity between nodes.

 In \cite{mounia2019laplacian}, we developed techniques for analyzing the \gls{test} of the regularized normalized Laplacian of \glspl{RGG} in the thermodynamic regime. The thermodynmic regime is a regime in which the average degree  of a vertex in the \gls{RGG} tends to a constant.  In particular, we showed that the {\gls{test}} of the regularized normalized Laplacian of an {\gls{RGG}} can be approximated by the \gls{test} of the regularized normalized Laplacian of a \gls{DGG} with nodes in a grid.  Then, in \cite{mounia2019laplacian} we provided an analytical approximation of the eigenvalues in the thermodynamic regime. In this work, we deepen the analyses in \cite{mounia2019laplacian} by using the obtained results on the spectrum of the \gls{RGG} to prove that the \gls{SD} of \glspl{RGG} is approximated by the space dimension $d$ in the thermodynamic regime.

The rest of this paper is organized as follows. 
In Section \ref{sec:spectral}, we define the \gls{RGG}  and we provide the results related to the \gls{test} of its regularized normalized Laplacian in the thermodynamic regime. In Section \ref{sec:spectralDimension}, we analyze the spectral dimension of \glspl{RGG} using the expression of the eigenvalues of the regularized normalized Laplacian and we validate the theoretical results on the eigenvalues and the spectral dimension of \glspl{RGG} by simulations. Finally, conclusions and future works are drawn in Section \ref{sec:conclusion}.

\section{Definitions and   Preliminary  Results on the Eigenvalues of \glspl{RGG}}
\label{sec:spectral}

Let us precisely define an \gls{RGG} denoted by $G(\mathcal{X}_{n}, r_{n})$ in this work. We consider a finite set $\mathcal{X}_{n}$ of $n$ nodes, $x_{1},...,x_{n},$ distributed uniformly and independently on the $d$-dimensional torus $\mathbb{T}^d \equiv [0, 1]^d$. Taking a torus $\mathbb{T}^d$ instead of a cube allows us not to consider boundary effects. Given a geographical distance $r_{n} >0 $, we form a graph by connecting two nodes $x_{i}, x_{j} \in \mathcal{X}_{n}$ if the $\ell_{p}$-distance between them is at most $r_{n}$, i.e., $\| x_{i}- x_{j} \|_{p} \leq r_{n}$ with $ p \in [1,\infty]$ (that is, either $p \in [1,\infty)$ or $p=\infty$). Here $\| . \|_{p}$ is the $\ell_{p}$-metric on $\mathbb{R}^d$ defined as
\begin{equation*}
\label{eqq1}
\| x_{i}- x_{j} \|_{p} =\left\{
    \begin{array}{ll}
       \left(  \sum_{k=1}^d \mid x_{i}^{(k)}-x_{j}^{(k)}\mid^p \right)^{1/p} &  \mathrm{for} \ \  p \in [1, \infty),\\ 
       \max\lbrace \mid x_{i}^{(k)}-x_{j}^{(k)} \mid, \ 1 \leq k \leq d \rbrace & \mathrm{for} \ \ p=\infty,
    \end{array}
\right. 
\end{equation*}
where the case $p=2$ gives the standard Euclidean metric on $\mathbb{R}^d$.

Typically, radius $r_{n}$ depends on $n$ and is chosen such that $r_{n}\rightarrow 0$ when $n \rightarrow \infty$.   A very important parameter in the study of the properties of graphs are the degrees of the graph vertices. The degree of a vertex in a graph is the number of edges connected to it.  The average vertex degree in $G(\mathcal{X}_{n}, r_{n})$ is given by \cite{penrose2003random}
\begin{equation*}
a_{n}=\theta^{(d)} nr_{n}^d,
\end{equation*} 
where $\theta^{(d)}=\pi^{d/2}/\mathrm{\Gamma}(d/2+1)$ denotes the volume of the $d$-dimensional unit hypersphere in $\mathbb{T}^d$, and $\mathrm{\Gamma}(.)$ is the Gamma function.  

 In \glspl{RGG}, we identify several scaling regimes based on the  radius $r_{n}$ or, equivalently, the average vertex degree, $a_{n}$. A widely studied regime is the connectivity regime, in which the average vertex degree $a_{n}$ grows logarithmically in $n$ or faster, i.e., $\mathrm{\Omega}(\log(n))$\footnote{The notation $f(n) =\mathrm{\Omega}(g(n))$ indicates that $f(n)$ is bounded below by $g(n)$ asymptotically, i.e., $\exists K>0$ and $ n_{o} \in \mathbb{N}$ such that $\forall n > n_{0}$ $f(n) \geq K g(n)$.}. In this work however, we pay a special attention to the thermodynamic regime in which the average vertex degree is a constant $\gamma$, i.e., $a_{n} \equiv \gamma$ \cite{penrose2003random}.

 In general, it is a challenging task to derive exact Laplacian eigenvalues for complex graphs and based on them to describe their dynamics. We remark  that for this purpose the use of deterministic structures is of much help. Therefore, we introduce an auxiliary graph called the \gls{DGG} useful for the study of the \gls{test} of \glspl{RGG}.

 The \gls{DGG} with nodes in a grid denoted by $G(\mathcal{D}_{n}, r_{n})$ is formed by letting $\mathcal{D}_{n}$ be the set of $n$ grid points that are at the intersections of axes parallel hyperplanes with separation $n^{-1/d}$, and connecting two points $x'_{i}$ and $x'_{j}$ $\in \mathcal{D}_{n}$ if $\|x'_{i}-x'_{j} \|_{p} \leq r_{n}$ with $p\in [1, \infty]$.  Given two nodes in $G(\mathcal{X}_{n}, r_{n})$ or in $G(\mathcal{D}_{n}, r_{n})$, we assume that there is always at most one edge between them. There is no edge from a vertex to itself. Moreover,  we assume that the edges are not directed. 

Diffusion on network structures is typically studied using the properties of suitably defined Laplacian operators. Here, we focus on studying the spectrum of the normalized Laplacian matrix.  However, to overcome the problem of singularities due to isolated vertices in the termodynamic regime, we use instead the regularized normalized Laplacian matrix proposed in \cite{avrachenkov2010improving}.  It corresponds to  the normalized Laplacian matrix on a modified graph constructed by adding auxiliary edges among all the nodes with weight $\frac{\alpha}{n}> 0$. Specifically, the entries of the regularized normalized Laplacian matrices of $G(\mathcal{X}_{n}, r_{n})$ and $G(\mathcal{D}_{n},  r_{n})$ in the thermodynamic regime are denoted by $\mathcal{L}$ and $\mathcal{L'}$, respectively, and are defined as 
\begin{equation*}
 \mathcal{L}_{i j}= \delta_{ij}-\dfrac{\chi [x_{i} \sim x_{j}]+\frac{\alpha}{n}}{ \sqrt{(\mathbf{N}(x_{i})+\alpha)(\mathbf{N}(x_{j})+\alpha)}},\ \ \ \  \mathcal{L'}_{i j}= \delta_{ij}-\dfrac{\chi [x'_{i} \sim x'_{j}]+\frac{\alpha}{n}}{ \sqrt{(\gamma'+\alpha)(\gamma'+\alpha)}},
\label{RW1} 
\end{equation*}
where, $\mathbf{N}(x_{i})$ and $\mathbf{N}(x'_{i})=\gamma'$ are the number of neighbors of the vertices $x_{i}$ and  $x'_{i}$ in $G(\mathcal{X}_{n}, r_{n})$ and $G(\mathcal{D}_{n}, r_{n})$, respectively. 
The term $\delta_{ij}$ is the Kronecker delta function. The term $\chi[x_{i}\thicksim x_{j}]$ takes unit value when there is an edge between nodes $x_{i}$ and $x_{j}$ in $G(\mathcal{X}_{n}, r_{n})$ and zero otherwise, i.e.,
\begin{equation*}
\label{eqq1}
\chi[x_{i}\thicksim x_{j}] =\left\{
    \begin{array}{ll}
        1, &  \| x_{i} - x_{j}\|_{p} \leq r_{n}, \ \ \  i \neq j \\
       0, & \mathrm{otherwise}.
    \end{array}
\right.
\end{equation*}

A similar definition holds for $\chi[x'_{i}\thicksim x'_{j}]$ defined over the nodes in $G(\mathcal{D}_{n}, r_{n})$.  

The matrices $ \mathcal{L}$  and $ \mathcal{L'}$ are symmetric, and consequently, their spectra consist of real eigenvalues. We denote by $\lbrace \lambda_{i}, i=1,..,n \rbrace$ and $\lbrace \mu_{i}, i=1,..,n \rbrace$ the sets of all real eigenvalues of the real symmetric square matrices $\mathcal{L}$  and $ \mathcal{L'}$ of order $n$, respectively. Then, the empirical spectral distribution functions of $ \mathcal{L}$ and $\mathcal{L'}$ are defined as

\begin{equation*}
F^{\mathcal{L}}_{n}(x)= \dfrac{1}{n} \sum\limits_{i=1}^{n}  \mathbf{1}_{\lambda_{i} < x},  \ \ \ \ \mathrm{and} \ \ \ \ \   F^{ \mathcal{L'}}_{n}(x)= \dfrac{1}{n} \sum\limits_{i=1}^{n} \mathbf{1}_{ \mu_{i} < x}.
\end{equation*}

In the following, we present a result which shows that $F^{ \mathcal{L'}} $ is a good approximation for $F^{ \mathcal{L}}$ for $n$ large in the thermodynamic regime by using the Levy distance between the two distribution functions.
\begin{definition}[\cite{taylor2012introduction}, page 257] 
Let $F^A$ and $F^B$ be two distribution functions on $\mathbb{R}$. The Levy distance $L(F^A, F^B)$ is defined as the infimum of all positive $\epsilon$ such that, for all $x \in \mathbb{R}$, 
\begin{equation*}
F^A(x-\epsilon) -\epsilon \leq F^B(x)\leq F^A(x+\epsilon)+\epsilon.
\end{equation*}
\end{definition}

\begin{lemma}[\cite{bai2008methodologies}, page 614]  
\label{Difference Inequality}
Let A and B be two $n$ $\times$ $n$ Hermitian matrices with eigenvalues $\lambda_{1},...,\lambda_{n}$ and $\mu_{1},...,\mu_{n}$, respectively. Then,
\begin{equation*}
L^{3}(F^A, F^B) \leqslant \dfrac{1}{n}tr(A-B)^2,
\end{equation*}
where $L(F^{A},F^{B})$ denotes the Levy distance between the empirical distribution functions $F^{A}$ and $F^{B}$ of the eigenvalues of $A$ and $B$, respectively. 
\end{lemma}

The following result provides an upper bound for the probability that the Levy distance between the distribution functions $F^{ \mathcal{L}}$ and $F^{\mathcal{L'}}$ is greater than a threshold $t$.
\begin{lemma}[\cite{mounia2019laplacian}]
\label{Lemma2} 
In the thermodynamic regime, i.e., for $a_{n} \equiv \gamma$ finite,  for $d \geq 1$, $p \in [1,\infty]$, and for every $t > \max\left[ \frac{4 \gamma'}{(\gamma'+\alpha)^2}, \frac{8 \gamma}{(\gamma+\alpha)^2}\right]$ as $n \rightarrow \infty,$ we get
\begin{equation*}
\lim_{n\to\infty} \mathrm{P} \left\lbrace L^{3} \left( F^{  \mathcal{L}}, F^{ \mathcal{L'}} \right) > t \right\rbrace =0.
\end{equation*}
\end{lemma}
In the thermodynamic regime,  Lemma \ref{Lemma2} shows that  $F^{  \mathcal{L'}}$ approximates $F^{  \mathcal{L}}$ with an error bound of $\max\left[\frac{4}{\gamma'}, \frac{8}{\gamma}\right] $ when $n \rightarrow \infty$ and $\alpha \rightarrow 0$, which in particular implies that the error bound becomes small for large values of $\gamma$. 

The following result provides approximated eigenvalues of the regularized normalized Laplacian of the $G(\mathcal{X}_{n}, r_{n})$ based on the structure of $G(\mathcal{D}_{n}, r_{n})$. 

\begin{lemma}[\cite{mounia2019laplacian}]
\label{eigenvalues}
For $d \geq 1$ and the use of the $\ell_{\infty}$-distance,  the eigenvalues of $\mathcal{L}$ are approximated as
 \begin{equation}
 \label{equation4}
\lambda_{m_{1},...,m_{d}} \approx 1 -\dfrac{1}{(\gamma'+\alpha)}  \prod_{s=1}^d   \dfrac{\sin(\frac{m_{s}\pi}{\mathrm{N}}(\gamma'+1)^{1/d})}{\sin(\frac{m_{s}\pi}{\mathrm{N}})}  + \dfrac{1-\alpha\delta_{m_{1},...,m_{d}}}{(\gamma'+\alpha)},
 \end{equation}
with $m_{1},...,m_{d}$ $\in$ $\lbrace 0,...\mathrm{N}-1 \rbrace$ and $\delta_{m_{1},...,m_{d}}=1$ when $m_{1},...,m_{d}=0$ otherwise $\delta_{m_{1},...,m_{d}}=0$. In (\ref{equation4}), $n=\mathrm{N}^{d}$ and $\gamma'=(2\left\lfloor \gamma ^{1/d}\right\rfloor+1)^d-1$ being $ \left\lfloor x \right\rfloor$ the integer part, i.e.,  the greatest integer less than or equal to $x$.
\\

 In particular, in the thermodynamic regime, under the  conditions described above and as $n \rightarrow \infty$, the eigenvalues of $\mathcal{L}$ are approximated as
\begin{equation}
\label{wow}
 \lambda_{w_{1},...,w_{d}}\approx1 -\dfrac{1}{(\gamma'+\alpha)}  \prod_{s=1}^d   \dfrac{\sin(\pi w_{s}^{1/d}(\gamma'+1)^{1/d})}{\sin(\pi w_{s}^{1/d})}  + \dfrac{1-\alpha\delta_{w_{1},...,w_{d}}}{(\gamma'+\alpha)},
\end{equation}
where $w_{s}=\frac{m_{s}^d}{\mathrm{n}}$ is in $\mathbb{Q} \cap [0, 1]$ and $\mathbb{Q}$ denotes the set of rational numbers. Similarly,  $\delta_{w_{1},...,w_{d}}=1$ when $w_{1},...,w_{d}=0$. Otherwise, $\delta_{w_{1},...,w_{d}}=0$.
\end{lemma}

Recall that the smallest eigenvalue $\lambda_{1}$ of a normalized Laplacian is always equal to zero, hence, $0 = \lambda_{1} \leq \lambda_{2} \leq... \leq \lambda_{n} \leq 2$. The second smallest eigenvalue $\lambda_{2}$ is called the Fidler eigenvalue. In the following, we show that the Fidler eigenvalue, $\lambda_{2}$ of the regularized normalized Laplacian of \glspl{RGG} goes to zero for large networks, i.e., $\lambda_{2} \rightarrow 0$ as $n\rightarrow \infty.$
 
 \begin{lemma}
The Fidler eigenvalue $\lambda_{2}$ of \glspl{RGG} in the thermodynamic regime is approximated as
\begin{equation}
\lambda_{2}\approx \dfrac{1}{(\gamma'+\alpha)}+1-(1+\gamma')^{\frac{d-1}{d}} \dfrac{\sin(\frac{\pi}{\mathrm{N}}(\gamma'+1)^{1/d})}{(\gamma'+\alpha)\sin(\frac{\pi}{\mathrm{N}})},
\end{equation}
where $n=\mathrm{N}^{d}$ and $\gamma'=(2\left\lfloor \gamma ^{1/d}\right\rfloor+1)^d-1.$  In particular, as $n \rightarrow \infty$, $\lambda_{2} \rightarrow 0.$
\end{lemma}
\begin{proof}
In general, the eigenvalues in (\ref{equation4}) are unordered,  but it is obvious that  the smallest eigenvalue is $\lambda_{0...0}$ and the next smallest one is $\lambda_{1, 0...0}=...=\lambda_{0...0,1}.$ Therefore, for $n=\mathrm{N}^{d}$, we have 
\begin{align}
\label{fidler}
\lambda_{2}& = \lambda_{1,0...0}  \nonumber \\
&\approx \lim_{m_{s}\to 0}\left( 1 -\dfrac{1}{(\gamma'+\alpha)}   \dfrac{\sin(\frac{ \pi}{\mathrm{N}}(\gamma'+1)^{1/d})}{\sin(\frac{ \pi}{\mathrm{N}})} \prod_{s=2}^d \dfrac{\sin(\frac{ m_{s}\pi}{\mathrm{N}}(\gamma'+1)^{1/d})}{\sin(\frac{m_{s} \pi}{\mathrm{N}})} + \dfrac{1}{(\gamma'+\alpha)}\right) \\
&= 1+\dfrac{1}{(\gamma'+\alpha)}-(1+\gamma')^{\frac{d-1}{d}} \dfrac{\sin(\frac{\pi}{\mathrm{N}}(\gamma'+1)^{1/d})}{(\gamma'+\alpha)\sin(\frac{\pi}{\mathrm{N}})}.
\end{align}

In large \glspl{RGG}, i.e., $n \rightarrow \infty$, we get
\begin{align*}
\lim_{n\to\infty} \lambda_{2}&  \approx  \lim_{n\to\infty}\left[ 1+\dfrac{1}{(\gamma'+\alpha)}-(1+\gamma')^{\frac{d-1}{d}} \dfrac{\sin(\frac{\pi}{\mathrm{N}}(\gamma'+1)^{1/d})}{(\gamma'+\alpha)\sin(\frac{\pi}{\mathrm{N}})}\right] \\
&=  1+\dfrac{1}{(\gamma'+\alpha)}-  1-\dfrac{1}{(\gamma'+\alpha)}= 0.
\end{align*}
\end{proof}

\section{Spectral Dimension of \glspl{RGG}}
\label{sec:spectralDimension}

In this section, we use the expression of the eigenvalues provided previously, in particular the eigenvalues in the neighborhood of $\lambda_{1}=0$ to find the spectral dimension $d_{s}$ of \glspl{RGG} in the thermodynamic regime.

 Recall that independently of the origin point, the spectral dimension can be defined through the return probability of the random walk, i.e., the probability to be at the origin after time $t$, $\mathrm{P}_{0}(t)$
 \begin{equation}
\label{spectral-dimension-definition}
d_{s}=-2 \dfrac{\mathrm{d}\ln \mathrm{P}_{0}(t)}{\mathrm{d}\ln (t)}.
\end{equation}

The return probability $\mathrm{P}_{0}(t)$ is related to the spectral density $\rho(\lambda)$ of the normalized Laplacian operator by a Laplace transform \cite{barrat2008dynamical}
\begin{equation*}
\mathrm{P}_{0}(t)=\int_{0}^{\infty} e^{-\lambda t}\rho(\lambda) \mathrm{d} \lambda,
\end{equation*}
so that the behavior of $\mathrm{P}_{0}(t)$ is connected to the spectral density $\rho(\lambda)$.  In particular its long time limit is directly linked to the behavior of $\rho(\lambda)$ for $\lambda \rightarrow 0$ \cite{touchette2005asymptotics}.

Before addressing the case of \glspl{RGG}, let us consider a simple example. In general, when the spectral dimension follows a power-law tail asymptotics, i.e.,  $\rho(\lambda) \sim \lambda^{\gamma},$ $\gamma >0$ for $\lambda \rightarrow 0$ then, for $t \rightarrow \infty$, we get
\begin{equation}
\label{powerlaw}
\mathrm{P}_{0}(t) \sim t^{-\gamma-1}.
\end{equation}

In a $d$-dimensional regular lattice, the low eigenvalue density is given by $\rho(\lambda) \sim \lambda^{d/2-1}$.  Then, the use of (\ref{powerlaw}) leads to the well known result
\begin{equation*}
\mathrm{P}_{0}(t)\sim t^{-d/2}.
\end{equation*}

 In this case, the spectral dimension $d_{s}$ can also be described according to the asymptotic behavior of the normalized Laplacian operator spectral density,  due to which it got its name
\begin{equation}
\label{spectral-dimension-eigenvalue}
\dfrac{d_{s}}{2}= \lim_{\lambda \rightarrow 0} \dfrac{\log(F(\lambda))}{\log(\lambda)},
\end{equation}
with $F(\lambda)$ being the empirical spectral distribution function of the normalized Laplacian.

Since the long time limit of the return probability $\mathrm{P}_{0}(t)$ or equivalently \gls{SD} is related to the eigenvalues density in the neighborhood of $\lambda_{1}$,   then in the following we analyze the behavior of the \gls{ED} of the regularized normalized Laplacian of \glspl{RGG} in a neighborhood of $\lambda_{1}$.
  
We have from (\ref{wow}) that the eigenvalues of the regularized normalized Laplacian  of \glspl{RGG} are approximated in the limit as

\begin{equation*}
\lambda(w)= \lambda_{w_{1},...,w_{d}}\approx1 -\dfrac{1}{(\gamma'+\alpha)}  \prod_{s=1}^d   \dfrac{\sin(\pi w_{s}^{1/d}(\gamma'+1)^{1/d})}{\sin(\pi w_{s}^{1/d})}  + \dfrac{1-\alpha\delta_{w_{1},...,w_{d}}}{(\gamma'+\alpha)}.
\end{equation*}

From Fig. \ref{fig_RGG1}(a), we can notice that the eigenvalues of the \gls{DGG} show a symmetry and the smallest ones are reached for small values of $w$. Additionally, for small and decreasing values of $w$, the eigenvalues of the regularized normalized Laplacian of \glspl{DGG} decrease. Therefore, in the following, we show that the empirical distribution of the eigenvalues in a neighborhood of $\lambda_{1}$, or equivalently, the eigenvalues for small values of $w$ of the regularized normalized Laplacian of \glspl{DGG} follow a power-law asymptotics. 

The eigenvalues of the regularized normalized Laplacian of \glspl{RGG} in a neighborhood of $\lambda_{1}$ are then approximated  by

\begin{equation*}
\lambda(w) \approx1 -\dfrac{1}{(\gamma'+\alpha)} \left[  \dfrac{\sin(\pi w^{1/d}(\gamma'+1)^{1/d})}{\sin(\pi w^{1/d})}\right]^d  + \dfrac{1-\alpha\delta_{w}}{(\gamma'+\alpha)},
\end{equation*}
for $w \rightarrow 0$.

The limiting distribution $F^{ \mathcal{L'}}(x)= \lim_{n \rightarrow \infty}F^{ \mathcal{L'}}_{n}(x)$ exists and is given by
\begin{equation}
\label{integral}
F^{ \mathcal{L'}}(x)=  \int_{0}^{1} \mathbf{1}_{(-\infty, x]}\left( \lambda(w)\right) dw= \int_{\lambda(w) \leq x} dw,
\end{equation}
where $\mathbf{1}_{S}(t)$ is the characteristic function on a set $S$ defined as

\begin{equation*}
\label{eqq1}
\mathbf{1}_{S}(t)=\left\{
    \begin{array}{ll}
   1  &  \mathrm{if} \ \  t \in S,\\ 
    0 & \mathrm{if} \ \ t \notin S.
    \end{array}
\right. 
\end{equation*}

The quantity $ \int_{\lambda(w)\leq x} dw$ is the measure of the set of all $w$ such that $\lambda(w) \leq x$.   Therefore, to compute (\ref{integral}), we only need to find the location of the points $w$ for which $\lambda(w) \leq x$ by solving $\lambda(w)=x$. However, the expression of $\lambda(w)$ includes the Chebyshev polynomials of the 2nd kind. In general there is no closed-form solution of the equation $\lambda(w)=0$. To find the spectral dimension, only small eigenvalues are of interest. Therefore,  for $\gamma'$ finite, we use Taylor series expansion of degree 2 around zero,  which is given by
 
\begin{equation*}
\ \lambda(w) \approx \dfrac{\pi^2}{6(\gamma'+\alpha)} w_{}^{2/d} (\gamma'+1)^{\frac{d+2}{d}}.
\end{equation*}

Fig. \ref{fig_RGG1}(b) validates this approximation and shows that it provides an accurate approximation.
Hence, to compute equation (\ref{integral}), we only need to find the location of the points $w$ for which $\lambda(w) \leq x$, by solving the new equation

\begin{equation*}
\lambda(w)=x  \Longleftrightarrow  \dfrac{\pi^2}{6(\gamma'+\alpha)} w_{}^{2/d} (\gamma'+1)^{\frac{d+2}{d}} =x.
\end{equation*}

By solving with respect to $x$, we obtain
\begin{equation*}
w=  \dfrac{6^{d/2} (\gamma'+\alpha)^{\frac{d}{2}}}{\pi^d(1+\gamma')^{(2+d)/2}}x^{d/2}.
\end{equation*}

Thus,  the limiting eigenvalue distribution $F^{ \mathcal{L'}}(x)$ for small $x$ is approximated by
\begin{equation}
\label{empirical}
F^{ \mathcal{L'}}(x) \approx \dfrac{6^{d/2} (1+\gamma'+1)^{-\frac{2+d}{2}}}{\pi^d}x^{d/2}.
\end{equation}

From (\ref{empirical}), it is apparent that the empirical distribution of the eigenvalues, $F^{ \mathcal{L'}}(x)$  of the regularized normalized Laplacian of a \gls{DGG} in a neighborhood of $\lambda_{1}$ follows a power-law tail asymptotics. 

Therefore, combining (\ref{spectral-dimension-eigenvalue}) and (\ref{empirical}), we get the spectral dimension $d_{s}$ of \glspl{RGG} in the thermodynamic regime as

\begin{align*}
\label{spectraldimension1}
d_{s} & \approx \lim_{x \rightarrow 0} \dfrac{2 \log\left(F^{ \mathcal{L'}}(x) \right)}{\log\left(x\right)}\\
&= \lim_{x \rightarrow 0}\dfrac{2 \log\left(  \dfrac{6^{d/2} (\gamma'+\alpha)^{\frac{d}{2}}}{\pi^d(1+\gamma')^{(2+d)/2}}x^{d/2} \right)}{\log\left( x \right)}\\
&=d.
\end{align*}

This result generalizes the existing work on the standard lattice in which it has already been shown that its spectral dimension, $d_{s},$ and its Euclidean dimension coincide.  In this work, we show that the spectral dimension in \glspl{DGG} with nodes in a grid is equal to the space dimension $d$ and is an approximation for the spectral dimension of the \gls{RGG} in the thermodynamic regime.  Thus, by taking a vertex degree in the \gls{DGG} corresponding to the standard lattice, we retrieve the result for the lattice.

\begin{figure*}[!t]
\centering
\subfloat[Eigenvalues of the \gls{DGG} for $\gamma=8$ (blue line) and $\gamma=28$ (red line), $d=1$.]{{\includegraphics[width=2in]{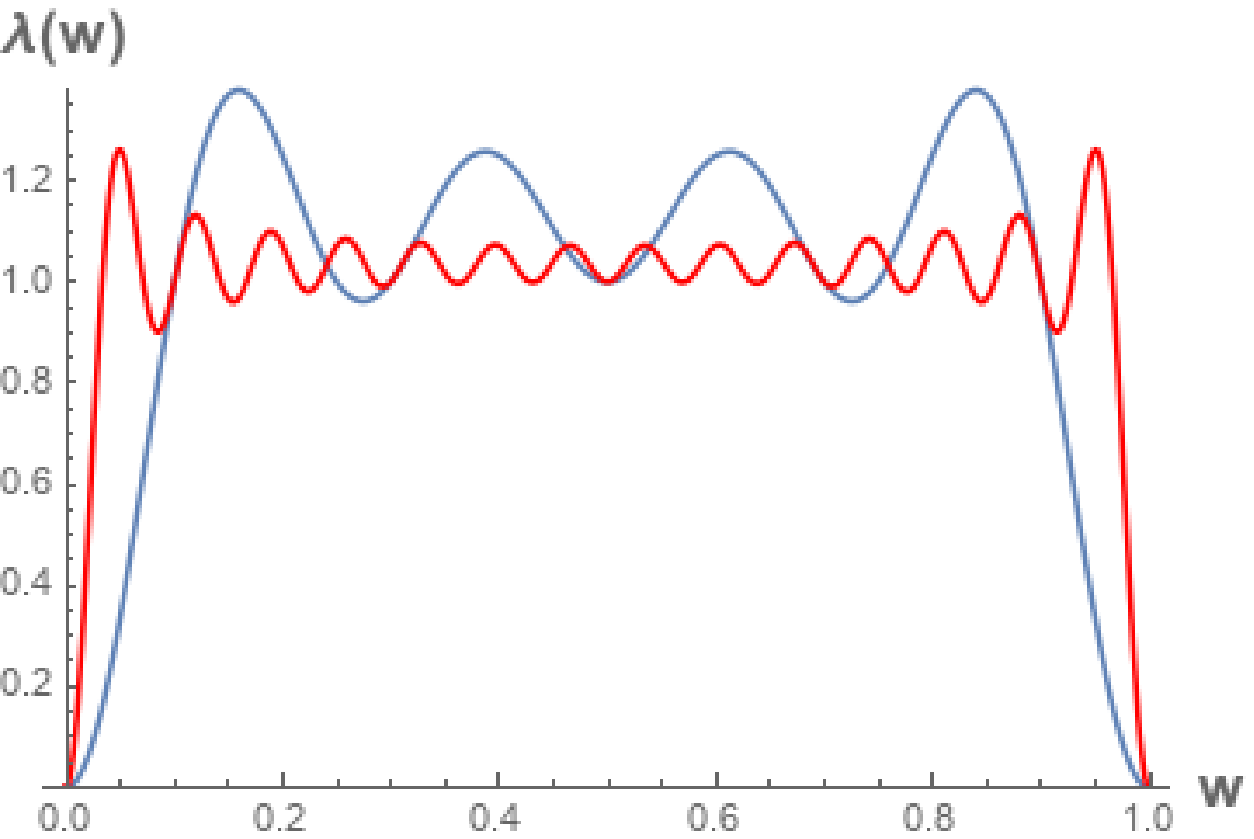}}}%
 \qquad
\subfloat[Comparison between the Analytical eigenvalues (blue) and its Taylor series approximation of degree 2 around zero (dashed orange) for $\gamma=12$, $\alpha=0.1$, $d=2$.]{{\includegraphics[width=2in]{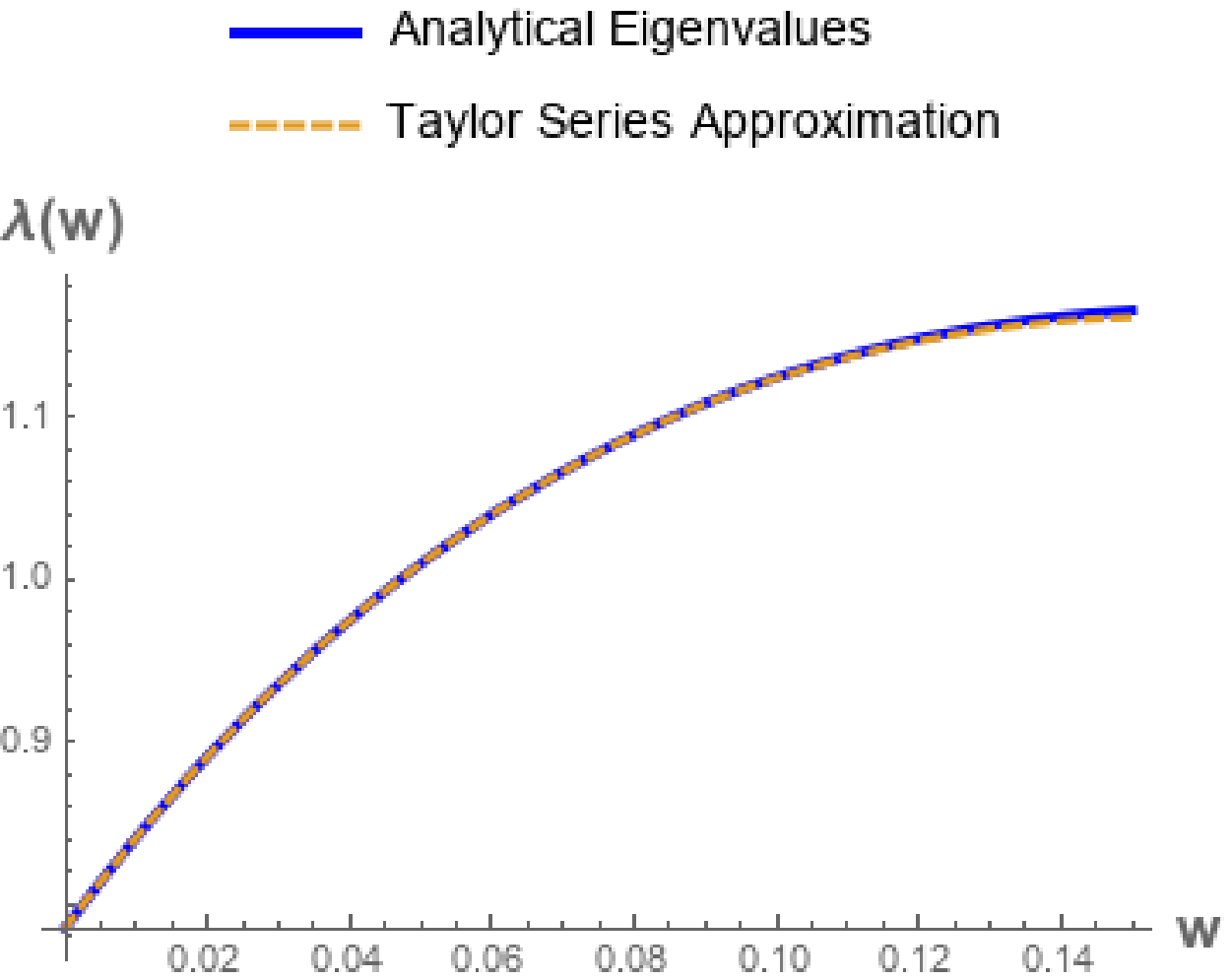}}}%

\caption{Eigenvalues of the \gls{DGG}.}
\label{fig_RGG1}
\end{figure*}

\section{Conclusions}
\label{sec:conclusion}
This work investigates the spectral dimension of \glspl{RGG} in thermodynamic regime. The notion of spectral dimension could serve as an estimator of the intrinsic dimension of the underlying geometric space in real problems where the geographical distance between nodes is a critical factor. We first show that the \gls{test} of the regularized normalized Laplacian of \glspl{RGG} can be approximated by the \gls{test} of the regularized normalized Laplacian of \glspl{DGG} as $n$ goes to infinity.  An analytical approximation of the eigenvalues of an \gls{RGG} regularized normalized Laplacian in a neighborhood of $\lambda_{1}$ is given. Then, using Taylor series expansion around zero, we approximate the empirical distribution of low-eigenvalues which is useful for the derivation of the spectral dimension in thermodynamic regime. The study shows that the spectral dimension, $d_{s}$ for \glspl{RGG} is approximated by $d$ in the thermodynamic regime.  As future works, we will analyze the spectral dimension and the eigenvalues of an \gls{RGG} in the connectivity regime.  In addition, we note that the result we derive in this paper on the spectral dimension for \glspl{RGG} may be useful in estimating the intrinsic dimension, a technique that might be used  to cope with high dimensionality data of networks modeled as \glspl{RGG}.

\section{\bf  Acknowledgement}
This research was funded by the French Government through the Investments for
the Future Program with Reference: Labex UCN@Sophia-UDCBWN.

\bibliographystyle{IEEEtran}
\bibliography{references}
\end{document}